\documentclass{amsart}
\usepackage{amsthm}
\usepackage{amsmath}
\usepackage{amssymb}
\usepackage[all]{xy}
\usepackage{enumitem}
\usepackage{verbatim}
\usepackage{framed}

\usepackage{color,soul}
\sethlcolor{yellow}

\def\F{\mathbb{F}}
\def\ff{\mathbb{F}}
\def\Tr{\mathrm{Tr}}
\def\frL{\mathfrak{L}}
\def\Supp{\mathrm{Supp}}
\def\zz{\mathbb{Z}}
\def\pp{\mathbb{P}}
\def\rr{\mathbb{R}}
\def\Frob{\mathrm{Frob}}

\newtheorem{theorem}{Theorem}[section]
\newtheorem{lemma}[theorem]{Lemma}
\newtheorem{prop}[theorem]{Proposition}
\newtheorem{cor}[theorem]{Corollary}

\theoremstyle{definition}
\newtheorem{definition}[theorem]{Definition}
\newtheorem{example}[theorem]{Example}

\theoremstyle{remark}

\subjclass[2010]{11G20 ,11L07 , 11T71}
\begin{document}

	\title{On the dimension of Algebraic-geometric trace codes}
	\author{P. Le \& S. Chetty}
	\date{}
	\maketitle

	\begin{abstract}
		We shall derive conditions in which a general formula for the dimension of $q$-ary trace codes induced by algebraic-geometric codes.  Significant to this result are several dimension reducing methods for the underlying functions spaces associated to $q$-ary trace codes.
	\end{abstract}

	\section{Introduction}
	Many good error correcting codes over a finite field can be constructed from other codes using the trace map.  More generally, given a code $C$ over a finite field $\mathbb{F}$, one can construct a subfield subcode by restriction (e.g. in the coordinates) to a subfield of $\mathbb{F}$, and Katsman and Tsfasman \cite{KT} prove that one can often obtain better parameters than trivial bounds guarantee. Precise lower bounds on the dimension of subfield subcodes have been given in, e.g., \cite{KT}, \cite{WIRTZ}, and \cite{STICHT90}. Delsarte's Theorem \cite{DELSARTE75} is used to describe subfield subcodes as trace codes and BCH-codes, classical and generalized Goppa codes, and alternant codes all can be realized as the dual of trace codes.
	
	Algebraic-geometric (AG) codes arise from the evaluation of the elements of an $\ff_{q^m}$-vector space of functions in a set of $\ff_{q^m}$ rational points on a curve $X$.  We shall consider trace codes associated to algebraic-geometric codes.  Conditions will be derived where one can determine the exact dimension of such codes. A key ingredient in the present work, as in \cite{HMO} and \cite{VERON} for certain classes of codes, is to understand the kernel of the trace map. Our use of Bombieri's estimate, following \cite{VLUGT91}, and consideration of a more general class of codes differs from the methods and setting of \cite{HMO} and \cite{VERON}.
	
	The main result, Theorem \ref{main}, is an extension of results that appear in \cite{VLUGT91}.  The original bound only applied for trace maps from the original field to the prime field.  We modify this to include trace maps to intermediate fields. Significant modifications of the original proof are needed to accommodate the more generalized trace in the execution of Bombieri's estimate for exponential sums.\cite{BOMBIERI66}.  The primary modification is summarised in Lemma \ref{bombcond}.
	
	For a general introduction on AG codes and trace codes, see \cite{STICHT93}.
	
	\section{Definition of Code and Main Result}
	
	\subsection{Background}

	Let $p$ be a prime number and $q=p^r$. Given a linear code $C$ of length $n$ over $\ff_{q^m}$, A trace code over $\ff_q$ is constructed from $C$ by applying the trace map from $\ff_{q^m}$ to $\ff_q$ coordinate-wise to the letters of the words of $C$.  This $q$-ary code is denoted $\Tr_{q^m/q} C$ or simply $\Tr C$ if the base fields in question are clear.
	
	Let $X$ be a geometrically irreducible, non-singular projective curve of genus $\mathfrak{g}$ defined over $\ff_{q^m}$.
	Consider $\ff_{q^m}(X)$ the $\ff_{q^m}$ vector space of functions on $X$.
	A divisor $G = \sum n_Q Q$, defined over $\ff_{q^m}$, may be split into two divisors $G^+$ and $G^-$ where $G^+ = \sum_{n_Q>0}n_Q Q$ and $G^- = \sum_{n_Q <0} n_Q Q$.  Hence $G = G^+ + G^-$.  The sum $\sum_i n_i$ of the coefficients of $G$ is called the degree of $G$, denoted $\deg(G)$.
	
	Define $\frL(G)$ to be the vector space of functions
	$$\frL(G) = \{ f \in \ff_{q^m}(X) \mid (f)+G \geq 0 \} \cup \{0\}.$$
	To generate a code from $\frL(G)$ we take a subset of $n$ distinct $\ff_{q^m}$-rational points away from the poles of $\frL(G)$:
	$$D:= \{P_1,\ldots,P_n\} \subseteq X(\ff_{q^m}) \setminus \Supp (G^+).$$
	For a divisor $G$ we denote the support of $G$ to be $\Supp(G):=\{Q \mid n_Q \neq 0\}$.
	For our purposes we will take $D= X(\ff_{q^m}) \setminus \Supp (G^+)$ the largest possible set.  Note that it is only necessary that $|D|>\deg(G)$.
	
	We define our AG code to be
		$$C:=C(D,G) = \{(f(P_1),\ldots, f(P_n)), f \in \frL(G) \}.$$
	When $2\mathfrak{g}-2<\deg(G)<n$, by Riemann-Roch we have
	$$\dim_{\ff_{q^m}}\frL(G) = \deg(G) +1-\mathfrak{g}.$$
	Under the assumptions of Riemann-Roch we have $\deg(G) \leq n$. Hence the dimension of $C$ as a $\ff_{q^m}$ vector space is also $k$.  In this way we identify $f \in \frL(G)$ with its image in $C$.
	
	An AG trace code is defined as the coordinate-wise application of the trace map
	$$\Tr C: = \{(\Tr_{q^m/q}(f(P_1)),\ldots, \Tr_{q^m/q}(f(P_n))), f \in \frL(G) \}.$$

	\subsection{Main Result}
	To state our main result we require the following construction based on the divisor $G$.
	For $r \in \rr$, let $[r]$ denotes the greatest integer function. Consider the divisor
	$$[G/q] := \sum_{n_Q >0} [n_Q/q] Q + \sum_{n_Q<0} n_Q Q.$$  That is, we are dividing the positive coefficients by $q$ and rounding down to the nearest integer.  We are in a sense dividing the pole part of $G$ by $q$.
	This construction will be useful in determining the kernel of the trace map.
	
	Section \ref{proof} is devoted to the proof of the following dimension formula for $\Tr C$:
	
	\begin{theorem}
		\label{main}
		Let $2\mathfrak{g}-2 \leq \deg([G/q])$ and $\deg(G)<n$.
		Consider the following two conditions:
		\begin{equation}
		\label{condition1}
		|Supp (G^-)| \leq 1,
		\end{equation}
		\begin{equation}
		\label{condition2}
		|X(\ff_{q^m})| > (2\mathfrak{g}-2 + \deg(G^+))q^{m/2}+Supp(G^+)(q^{m/2}+1).
		\end{equation}
		Under these conditions we have an exact formula for the dimension:
		$$\dim_{\ff_q}\Tr C = m(\deg(G) - \deg([G/q]))+\delta,$$
		where
		$$\delta = \begin{cases}
		1 & \mathrm{if~} |\Supp(G^-)|=0,\\
		0 & \mathrm{otherwise}.
		\end{cases}$$
		
		
	\end{theorem}
	
	If $q$ is a prime number this theorem reduces to the main result of \cite{VLUGT91}.  Theorem \ref{main} is applicable for a more generalized trace when $q = p^r$ for some prime $p$ and $r>1$.  In this setting complications arise in the dimension reducing argument using Bombieri's estimate.  Bombieri's estimate alone does not collapse the dimension of the kernel of the trace map enough.  We have addressed these complications with the addition of a degree argument that shows the kernel can be reduced in a way that aligns with the result in \cite{VLUGT91}.
	\subsection{Examples}
	
	\begin{example}
		Let $q=p^r$.  Consider an elliptic curve $E$ defined over $\ff_{q^m}$.  A formula for counting points on E is given by:
		$$E(\ff_{q^m}) = q^m+1 - \pi^{m} - \overline{\pi}^{m}$$
		where $\pi \overline{\pi} = q$  and $\pi+ \overline{\pi} = a_p$, the linear coefficient of the numerator of an associated zeta function as described in \cite[page 301]{CIMNT}.
		
		For $G = k P_{\infty}$ we see that condition (\ref{condition2}) is
		$$ q^{m/2} - \frac{\pi^{m} + \overline{\pi}^{m}}{q^{m/2}}>k.$$
		Assuming this, Theorem \ref{main} states:
		For any $D$ such that $|D|>\deg(G)$ we have:
		$$\dim_{\ff_q} \Tr C(D,G) = m(k - [k/q]).$$
	\end{example}
	
	\begin{example}
		For a smooth projective curve $X$ defined over $\ff_{q^m}$, let $G = k P_{\infty}$ for  some positive integer $k$.  Using the the Hasse-Weil bound we have
		$$||X(\ff_{q^m})|-(q^m+1)| \leq 2\mathfrak{g}  q^{m/2}.$$
		By condition (\ref{condition2}) of Theroem \ref{main} we must also require the inequality
		$$|X(\ff_{q^m})| > (2\mathfrak{g}-2+k)q^{m/2} + (q^{m/2}+1).$$
		Combining these two inequalities, we see that (\ref{condition2}) is satisfied when
		$$q^{m/2}-4\mathfrak{g}+1>k.$$
		Using Theorem \ref{main} we obtain the following:
		\begin{cor}
			For $X$ a smooth projective curve over $\ff_{q^m}$ and $G=k P_{\infty}$,
			if $2\mathfrak{g}-2\leq [k/q]$ and  $k<\min(n,q^{m/2}-4\mathfrak{g}+1)$ then
			$$\dim_{\ff_q} \Tr C = m(k - [k/q])+1.$$
		\end{cor}
	\end{example}
	
	\begin{example} This is a generalization from an example in \cite{VLUGT91}.
		Let $q=p^r$, $X = \pp^1$ and $G= (g)_0 - P_{\infty}$ where $(g)_0$ is the zero divisor of a polynomial $g(z) \in \ff_{q^m}[z]$ which has no zeros in $\ff_{q^m}$.  Denote the number of different zeros of $g(z)$ by $s$.  Furthermore, we take $D=\sum_{x \in \ff_{q^m}} P_{x}$.  From condition (\ref{condition2}) we obtain the inequality:
		$$\deg(g(z))+s < \frac{q^m+1}{\sqrt{q^m}}+2.$$
		Write $g(z)=g_1^q g_2$, with $g_1(z),g_2(z) \in \ff_{q^m}[z]$ of degrees $r_1, r_2$ respectively, and $g_2(z)$ $q$-th power free.  With sufficiently many points as above, applying Theorem \ref{main} we have:
		$$\dim_{\ff_q} \Tr C(D,G) = m((q-1)r_1+r_2).$$
	\end{example}
	
	\section{Proof of Main Result}
	\label{proof}
	
	First note $C$ is a vector space over $\ff_{q^m}$ and $\Tr C$ is a vector space over $\ff_q$.  From this we  have the bound
	$$\dim_{\ff_{q^m}}C \leq \dim_{\ff_q}\Tr(C) \leq m(dim_{\ff_{q^m}}C).$$
	Using the $\ff_q$-linearity of the trace we have an exact sequence
	$$0 \rightarrow K \rightarrow C \rightarrow \Tr C \rightarrow 0$$
	where $K$ is the kernel, an $\ff_q$-linear subspace of $C$.  Hence
	\begin{equation}
	\label{eqnmain}
	m\dim_{\ff_{q^m}}C - \dim_{\ff_{q}} K = \dim_{\ff_q}\Tr(C).
	\end{equation}
	Therefore we can obtain the dimension of $\Tr C$ by first determining $\dim_{\ff_q} K$.  In practice, this is difficult.  Consider the subspace of K:
	
	$$E := \{ f = h^q-h  \mid f \in \frL(G), h \in \ff_{q^m}(X)\}.$$
	We will determine a sufficient condition (condition \ref{condition2} of Theorem \ref{main}) when $E=K$ using Bombieri's estimate and a degree argument. But to make this useful we first find conditions to determine the dimension of $E$.  First we will determine a sufficient condition (condition \ref{condition1} of Theorem \ref{main}) to determine the dimension of $E$.

	\subsection{Dimension of $E$}
	For $f = h^q-h \in \frL(G)$, by definition $(h^q-h) +G \geq 0$.  Counting multiplicity, each pole in $h$ corresponds to $q$ poles in $f$.   Hence, for $h \in \frL([G/q])$ we have $h^q-h \in \frL(G)$.

	Consider the map $\phi:\frL([G/q]) \to E$ where $\phi(h)=h^q-h$.  By examining degrees we can determine the kernel is $\ff_q \cap \frL([G/q])$. Note that for a general $G$ the map $\phi$ is not surjective.  
	
	\begin{lemma}
		\label{phisurj}
		When $|Supp (G^-)| \leq 1$ the map $\phi$ is surjective.
	\end{lemma}
	\begin{proof}
		
		Recall $[G/q]$ only changes the positive coefficients of $G$ and does not change $G^-$.  When $G^- = \emptyset$, there is no restriction on zeros in $\frL(G)$.  Therefore in this case $\phi$ is onto.
		
		When $|Supp(G^-)| = 1$ then $G^-=n_p P$ for some point $P \in X(\ff_{q^m})$ and negative integer $n_p$.  Every function in $\frL([G/q])$ must have a zero at $P$. In the factorization $h^q-h=\prod_{b\in \ff_q}(h-b)$, this zero must occur in at least one factor $h-b$.  Though $h$ may not be in $\frL([G/q])$, there will always exist some $b \in \ff_q$ such that $h-b \in \frL([G/q])$.  Observe $h^q-h = (h-b)^q-(h-b) = \phi(h-b)$.  Hence in this case $\phi$ is onto.
	\end{proof}
	
	If $G^- = \emptyset$ then the kernel of $\phi$ is $\ff_q$.  If $G^- \neq \emptyset$ then $\phi$ is injective.  Therefore for $\delta$ defined in Theorem \ref{main}, $\delta = \dim_{\ff_q} \ker \phi$.  Using Lemma \ref{phisurj} we have the following proposition.
	
	\begin{prop}
		\label{RRbig}
		If $| Supp(G^-)| \leq 1$, then the sequence
		
		$$
		\xymatrix
		{
			0 \ar[r] & \ff_q \cap \frL([ G/q ]) \ar[r] & \frL([ G/q ]) \ar[r]^{\phi} & E \ar[r] & 0
		}
		$$
		is exact.  Therefore we have a dimension formula for $E$:
		$$\dim_{\ff_q} E = \dim_{\ff_q} \frL[G/q]- \dim_{\ff_q} (\ff_q \cap \frL[G/q]).$$
	\end{prop}
	
	Though $\phi$ may not exact we still have a dimension bound:
	$$\dim_{\ff_q} E \geq \dim_{\ff_q} \frL[G/q]- \dim_{\ff_q} (\ff_q \cap \frL[G/q]).$$
	
	Note in \cite{VLUGT91}, a similar result is obtained with the use of cohomology and other auxillary constructions.
	
	\subsection{Bombieri's Estimate}
	Our primary tool for determining when $K=E$ is a bound developed by Bombieri \cite{BOMBIERI66}.
	\begin{theorem}[Bombieri's estimate]
		\label{BOMBIERI}
		Let $X$ be a complete, geometrically irreducible, nonsingular curve of genus $\mathfrak{g}$, defined over $\ff_{q^m}$.  Let $f \in \ff_{q^m}(X), f \neq h^p-h$ for $h \in \overline{\ff_{q}}(X)$, with pole divisor $(f)_{\infty}$ on $X$. Then
		$$\left| \sum_{P \in X(\ff_{q^m} )\setminus (f)_{\infty}} \zeta_p^{\Tr_{q^m/p}(f(P))}\right| \leq (2\mathfrak{g}-2+t+\deg(f)_{\infty})q^{m/2}.$$
		where  $\zeta_p = \exp(2\pi i/p)$ is any primitive $p$-th root of unity and $t$ is the number of distinct poles of $f$ on $X$.
	\end{theorem}
	
	Let $\overline{E} = \{f \in K \mid f=h^p-h \text{ for some }h\in \overline{\ff_{q}}(X)\}$.  Observe this is the subspace of $K$ which  prevents the use of Bombieri's estimate.
	
	\begin{lemma}
	\label{Esubset}
		$E\subseteq \overline{E}$.
	\end{lemma}
	\begin{proof}
		Recall $q=p^r$.  Therefore:
		$$\begin{array}{rcl}
		g^q-g &=& g^{p^r} - g\\
		&= &((g)^{p^{r-1}}+ \ldots +(g))^p - ((g)^{p^{r-1}}+ \ldots + (g))
		\end{array}$$
		Let $h=(g)^{p^{r-1}}+ \ldots +(g)$ and we see clearly that $g^q-g = h^p-h$.
	\end{proof}
	
	\begin{lemma}
		\label{equalsets}
		For each $g \in \overline{\ff_q}(X)$, there exists an $h \in \ff_{q^m}(X)$ and $c \in \ff_{q^m}$ such that $g^p-g=h^p-h+c$. Furthermore,
			$$\overline{E}  \subseteq  \{ f \in \ff_{q^m}(X) ~|~ f = h^p-h+c \mathrm{~for~some~} h \in \ff_{q^m}(X),c \in \ff_{q^m}\}. $$	
	\end{lemma}
	\begin{proof}
		Suppose there is an $f \in \ff_{q^m}(X)$ and an $h \in \overline{\ff_q}(X)$ such that $f=h^p-h$.  Take $\sigma = \Frob_{q^m}$, the $q^m$ Frobenius endomorphism on $\overline{\ff_{q^m}}$.  Since $\sigma(f) = f$ we may rework this so that $(\sigma(h)-h)^p = \sigma(h)-h$.  By considering the order of poles of $\sigma(h)-h$ we determine that $\sigma(h)-h$ must be a constant $a \in \ff_p$.  Let $a = b^{q^m}-b$ for some $b \in \overline{\ff_q}$.  Then $\sigma(b)=b+a$.  Also $\sigma(h-b)=h+a-(b+a) = h-b$.  Therefore $h-b \in \ff_{q^m}(X)$.  Let $h_1=h-b$.  Observe $f- b^p+b = h_1^p-h_1$.  Also $\sigma(b^p-b) = b^p-b$ so $b^p-b \in \ff_{q^m}$.  Therefore $f=h_1^p-h_1+b^p-b$.
	\end{proof}	
	
	For each $f \in K$, by definition $\Tr(f(P)) =0$. Observe that if $f \in K \setminus \overline{E}$ then $f$ satisfies the conditions of Bombieri's Estimate.  Hence, $\zeta_p^{\Tr(f(P))}=1$ for each $P$.  For such $f$, each term of the sum in the left-hand-side in Theorem \ref{BOMBIERI} contributes $1$.  This is a total contribution of $|X(\ff_q) \setminus (f)_{\infty}|$.  Hence for $f \in K \setminus \overline{E}$ we have
	
	$$\left| \sum_{P \in X(\ff_{q^m} )\setminus (f)_{\infty}} \zeta_p^{\Tr_{q^m/p}(f(P))}\right| = |(X(\ff_{q^m}) \setminus (f)_{\infty})| \leq (2\mathfrak{g}-2 + t + \deg(f)_{\infty}) q^{m/2}.$$
	
	Observe $t \leq \#\Supp(G^+)$ and $\deg(t)_{\infty} \leq \deg(G^+)$.  Using these two inequalities we obtain a more general bound:
	$$|X(\ff_{q^m}) | \leq (2\mathfrak{g}-2  + \deg(G^+))q^{m/2}+\#\Supp(G^+)(q^{m/2}+1).$$
	\begin{prop}
		\label{proppointbound}
		If $$|X(\ff_{q^m}) | > (2\mathfrak{g}-2  + \deg(G^+))q^{m/2}+\#\Supp(G^+)(q^{m/2}+1)$$ then $K=\overline{E}$.
		
	\end{prop}
	
	Note that the condition presented in Proposition \ref{proppointbound} is exactly condition \ref{condition2} from Theorem \ref{RRbig}.
	\begin{proof}
		Suppose such an inequality holds and $K \supsetneq \overline{E}$.  Then there is an element of $f \in K$ where $f \neq h^q-h$ for any $h \in \overline{\ff_q}(X)$.  Therefore we may apply Bombieri's estimate.  This is a contradiction.
	\end{proof}
	
	\subsection{$E$ and $\overline{E}$}
	Recall the definitions of $E$ and $\overline{E}$:

	$$E := \{ f = h^q-h  \mid f \in \frL(G), h \in \ff_{q^m}(X)\}.$$
	$$\overline{E} = \{f \in K \mid f=h^p-h \text{ for some }h\in \overline{\ff_{q}}(X)\}$$
	In  the case presented in \cite{VLUGT91} Vlugt had $\overline{E}=E$.  In the current more general case, the above argument provides conditions forcing all elements of $K$ to be of the form $h^p-h$, for $h\in\overline{\ff_q}(X)$.  However there may still be elements of this form that are not of form $g^q-g$, with $g\in \ff_{q^m}(X)$.  We will show that this is not the case and that condition \ref{condition2} of Theorem \ref{main} is sufficient to force $K=E$. To show this, it will be useful to develop our understanding of the interplay of $K$, $\overline{E}$ and $E$ and the nature of the degree of functions therein.
	
	As we saw previously in Lemma \ref{Esubset}, elements of the form $g^q-g$ can also be written in the form $h^p-h$. Also observe for any $f \in K$ and $y \in \ff_q$, the function $yf \in K$.  Furthermore, for any $f \in E$ and $y \in \ff_q$, $yf \in E$.   Consider the following:

	\begin{definition}
		Let $D(f)$ be the elements $y \in \ff_q$ such that $yf = h^p-h$ for some $h \in \overline{\ff_q}(X)$.
	\end{definition}
	\noindent
	
	Observe for $f \in K$, when $|D(f)|< q$ there is a $y$ such that $yf \in K \setminus \overline{E}$. Therefore, to prove Proposition \ref{bombcond} it will suffice to show that $|D(f)| < q$ for some $f \in K$.
	
	\begin{prop}
		Let $f,g \in K$, $t \in \F_p$ and $y \in D(f)$. We immediately have the following:
		
		\begin{enumerate}
			\item $f+g \in D(f)$
			\item $ty \in D(f)$
			\item $D(f)$ is a subgroup of $\F_q$ under addition.
			
		\end{enumerate}
	\end{prop}

	\begin{prop}
		\label{Dvector}
		If $D(f)=\ff_q$ and $D(g) = \ff_q$ then $D(af+bg)=\ff_q$ for each $a,b \in \ff_q$.
	\end{prop}

	\begin{lemma}
	\label{Dbound}
		
	Let  $f=h^p-h$ for some $h \in \overline{\F_q}(X)$ and $D(f)\neq \{0\}$.  Then
		$$|D(f)| \leq p|D(h)|$$
	\end{lemma}
		
	\begin{proof}
			
		Let $y \in D(f)$. Then $yf = g^p-g$ for some $g \in \overline{\F_q}(X)$.  Hence
			$$yf = g^p-g = yh^p-yh = (y^{1/p}h)^p -(y^{1/p}h) + (y^{1/p}h) - yh.$$
		Therefore
			$$(y^{1/p}-y)h = (g^p-g) -((y^{1/p}h)^p -(y^{1/p}h))$$
		Therefore $(y^{1/p}-y)$ is in $D(h)$. Hence for every $x, y \in D(f)$, $x^{1/p}-x$ and $y^{1/p}-y$ are in $D(h)$. Observe when
			$$x^{1/p}-x = y^{1/p}-y$$
			$$(x-y)^{1/p} = (x-y)$$
			$$(x-y)^p = (x-y)$$
		The only elements of $\F_q$ equal to their own $p^{\text{th}}$ power are elements of $\F_p$. Hence $x=y+t$ for some $t \in \F_p$. From this we see that $D(f)/\F_p$ can be identified with a subgroup of $D(h)$.  Hence
			$$|D(f)| \leq p|D(h)|.$$
	\end{proof}

	\begin{definition}
		For $f \in \ff_{q^m}(X)$, define the $p$-linear degree of $f$, denoted $e(f)$, to be the largest possible integer such that $f =a_c+a_0 g + a_1 g^p+ \ldots + a_{e(f)} g^{p^{e(f)}}$, where $a_c, a_0,\ldots, a_{e(f)}  \in \ff_{q^m}$, $g \in \ff_{q^m}(X)$.
	\end{definition}

	We immediately have the following  properties of $e(f)$:
	
	\begin{prop}~\
		
		\begin{enumerate}
			\item For $f \in \overline{E}$ we have $e(f) \geq 1$.
			\item For $g \in E$ we have $e(g) \geq r$.
			\item For $a \in \ff_{q^m}^*, b \in \ff_{q^m}$ we have $e(f) =e(af+b)$.
		\end{enumerate}
		
	\end{prop}
	
	\begin{prop}
	\label{Dbound2}
		For $f \in \ff_{q^m}(X)$ such that $e(f) = 0$ or $f=h^p-h$ for some $h\in \overline{\F_q}(X)$, we have the inequality:
		$$|D(f)| \leq p^{e(f)}.$$
	\end{prop}
	
	\begin{proof}

		We proceed by induction. Suppose $e(f)=0$.  By Lemma \ref{equalsets} we have $yf \neq h^p-h+c$, for any $y \in \ff_{q^m}^*$, $h \in \ff_{q^m}(X)$, and $c\in \ff_{q^m}$.  Therefore $D(f) = \{0\}$ and $|D(f)|=1 \leq p^{e(f)} = p^0 = 1$.
		
		Now consider a positive integer $k$ and $f$ such that $e(f) = k$.  Without loss of generality we may assume that $|D(f)| > 1$.  Hence we may also assume without loss of generality that $f=h^p-h$ for some $h\in \ff_{q^m}(X)$. It follows that $e(f) \geq 1 + e(h)$.  Since $e(f)=k$ we have  $e(h) <k$.  By the inductive hypothesis $|D(h)| \leq p^{e(h)}$.  Combining this with Lemma \ref{Dbound} we have: 
		$$|D(f)| \leq p|D(h)|=p^{e(h)+1} \leq p^{e(f)}.$$

	\end{proof}

	\begin{cor}
		\label{corDbound}
		If $|D(f)| =q=p^r$ then $e(f)$ $\geq r$.
	\end{cor}

	\begin{prop}
		\label{bombcond}
		Suppose $$|X(\ff_{q^m}) | > (2\mathfrak{g}-2  + \deg(G^+))q^{m/2}+\#\Supp(G^+)(q^{m/2}+1).$$ If $K \neq E$, then $K \setminus \overline{E}$ is nonempty.
	\end{prop}
	\begin{proof}
	
	Suppose $K \neq E$ and $D(f)=\ff_q$ for each $f\in K \setminus E$.  Such an $f \in \ff_{q^m}(X)$ cannot be constant.  Choose $f \in K \setminus E$ with the least number of poles. In other words, $\deg(f)_{\infty}$ is minimal and positive.  By the assumption on $|X(\ff_{q^m})|$, we may apply Corollary \ref{corDbound}, so there is some $l \in \zz_{\geq 0}$, $h\in \ff_{q^m}(X)$ and $a_c, a_1,\ldots,a_{r+l} \in \ff_{q^m}$ such that
	$$f = a_{r+l} h^{p^{r+l}} + a_{r+l-1} h^{p^{r+l-1}} + \ldots  +a_{1}h+a_{c}.$$
	We may rewrite this as
	$$f= f_E+f_1$$
	where
	$$f_E = (a_{r+l}h^{p^l})^q -( a_{r+l}h^{p^l})$$
	$$f_1 = ( a_{r+l}h^{p^l})+a_{r+l-1} h^{p^{r+l-1}} + \ldots  +a_{1}h+a_{c} \in \ff_{q^m}(X).$$
	Observe $f_E \in K$ and $D(f_E)= \ff_q$.  Hence $f_1 = f-f_E \in K$. By Proposition \ref{Dvector},  $D(f_1) = \ff_q$.  But
	$$\deg(f_1)_{\infty} \leq p^{r+l-1} \cdot \deg(h)_{\infty}$$
	and
	$$\deg(f)_{\infty} = p^{r+l} \deg(h)_{\infty}.$$
	This contradicts the choice of an $f$ with minimal poles.  Hence when $K\neq E$ we can always choose an $f \in K$ not of the form $h^p-h$.  
	\end{proof}
	
	\begin{proof}[Proof of Theorem $\ref{main}$]
	
	Let  $2\mathfrak{g}-2 \leq \deg([G/q])$, $\deg(G)<n$ and also assume condition \ref{condition1} and condition \ref{condition2}.
	
	By condition \ref{condition2}, Proposition \ref{proppointbound} and Proposition \ref{bombcond} we see that $K=E$.  Then, using  condition \ref{condition1} and \ref{RRbig} we compute the dimension of $E$.  Applying this to equation \ref{eqnmain} we obtain Theorem \ref{main}, a dimension formula for algebraic-geometric trace codes, as desired.
	\end{proof}
	
	\section*{Acknowledgement}
	The authors would like to thank Daqing Wan for his generous comments and suggestions.

	\bibliographystyle{amsplain}
	\bibliography{ple}

	\newcommand{\tw}{.5\textwidth}
	\begin{minipage}[b]{\tw}
		~\\
		
		Phong Le\\
		\textsc{Goucher College\\
			Towson, MD 21211}\\
		{\tt phong.le@goucher.edu}\\
		
	\end{minipage}
	
	\begin{minipage}[b]{\tw}
		~\\
		
		Sunil Chetty\\
		\textsc{College of Saint Benedict and Saint John's University\\
			Saint Joseph, MN}\\
		{\tt schetty@csbsju.edu}
	\end{minipage}
	
\end{document}